\documentclass{amsart}
\usepackage{amsmath}
\usepackage{graphicx}
\usepackage[dvipsnames]{xcolor}
\usepackage{hyperref}
\usepackage{amssymb,bbold}
\usepackage{enumitem}
\hypersetup{colorlinks=true,citecolor=blue,linkcolor=cyan}
\newtheorem{thm}{Theorem}[section]
\newtheorem{cor}[thm]{Corollary}
\newtheorem{lem}[thm]{Lemma}
\newtheorem{ques}[thm]{Question}
\newtheorem{prop}[thm]{Proposition}
\theoremstyle{definition}
\newtheorem{defn}[thm]{Definition}
\theoremstyle{remark}
\newtheorem{rem}[thm]{Remark}
\numberwithin{equation}{section}

\newcommand{\F}{\mathcal{F}}
\newcommand{\Ff}{\mathbb{F}}

\newcommand{\Hh}{\mathcal{H}}
\newcommand{\Tr}{\mbox{Tr}}

\newcommand{\Res}{\mbox{Res}}

\newcommand{\WP}{\raisebox{.15\baselineskip}{\Large\ensuremath{\wp}}}

\newlist{primenumerate}{enumerate}{1}
\setlist[primenumerate,1]{label={(\arabic*$'$})}
\begin{document}

\title[Moments]{Moments of Traces of Frobenius of Higher Order Dirichlet $L$-functions over $\mathbb{F}_q[T]$}%

\author{Patrick Meisner}
\address{KTH Royal Institute of Technology}
\email{pfmeisner@gmail.com}

\thanks{The author was supported in part by the Verg foundation.}

\begin{abstract}

We study the moments of $\Tr(\Theta_\chi)$ as $\chi$ runs over Dirichlet characters defined over $\Ff_q[T]$ of fixed order $r$. In particular, we show that after an appropriate normalization, the $q$-limit of the power sum moments behave like the power sum moments of the group of unitary matrices multiplied by a weight function.

\end{abstract}
\maketitle

\section{Introduction}

Fix a prime power $q$ and let $\Ff_q[T]$ be the ring of polynomials over the finite field of $q$ elements. Then an $L$-series defined over $\Ff_q[T]$ is a series of the form
\begin{align}\label{LSerDef}
L(u) := \sum_{F \mbox{ monic}} \frac{a_L(F)}{|F|^s} = \sum_{F \mbox{ monic}} a_L(F)u^{\deg(F)}
\end{align}
where $|F|=q^{\deg(F)}$, $u=q^{-s}$ and $a_L(F)\in \mathbb{C}$. If the $a_L(F)$ are chosen such that they capture arithmetic or geometric structure then we call $L$ an $L$-\textit{function}. In this case the properties of this structure (ex. how many primes exhibit this structure) can be understood by analyzing the zeroes of the $L$-functions.

The connection between $L$-functions and primes can be seen through the fact that $L$-functions typically admit an Euler product
\begin{align}\label{EulProd}
    L(u) = \prod_P \left( 1+a_L(P)u^{\deg(P)} + a_L(P^2)u^{2\deg(P)} + \cdots \right)
\end{align}
that converges for $|u|<q^{-1}$ where the product is over all monic, irreducible (prime) polynomials in $\Ff_q[T]$.

The Riemann Hypothesis, which has been proven for a large class of $L$-functions defined over $\Ff_q[T]$ by Weil \cite{wei}, states in part that $L(u)$ will be a polynomial, all of whose roots lie on the half-line $Re(s)=\frac{1}{2}$ or $|u|=q^{-1/2}$. As a result, we may find a unitary matrix $\Theta_L$, called the Frobenius of the $L$-function, such that
\begin{align}\label{FrobDef}
L(u) = \det(1-\sqrt{q}u\Theta_L).
\end{align}

In particular, the zeroes of the $L$-function correspond to the eigenvalues of $\Theta_L$.  Thus the statistics of zeroes of the $L$-function are equivalent to the statistics of the eigenvalues of its Frobenius. 

Katz and Sarnak \cite{KS} developed a philosophy that claims that for every ``natural" family of $L$-functions, $\mathcal{F}$, the Frobenii will equidistribute in some compact matrix Lie group, $G$, as $q$ tends to infinity. Here $G$ is called the \textit{monodromy group} or \textit{symmetry type} of the family. 

More specifically, for any set $S$ and any function $\phi$ on $S$ we use the notation $\left\langle \phi(s)\right\rangle_S := \frac{1}{|S|}\sum_{s\in S} \phi(s)$. Then for any continuous class function of $G$, $f$, the Katz-Sarnak philosophy predicts that there exists a compact matrix Lie group, $G$, such that 
\begin{align}\label{KSEqui}
    \lim_{q\to\infty} \left\langle f(\Theta_L)\right\rangle_{\mathcal{F}}  = \int_{G} f(M) dM
\end{align}
where $dM$ is the corresponding Haar measure of $G$. Further, they predict that $G$ will typically be one of the classical compact Lie groups: $U$, the unitaries; $USp$, the unitary symplectics; $O$, the orthogonal; $SO(even)$, the special orthogonal of even dimension; or $SO(odd)$, the special orthogonals of odd dimension. Sarnak, Shin and Templier \cite{SST} give an explicit description of conditions for a family of $L$-function to exhibit in order to have each symmetry type.


An important subset of continuous class functions are the mixed power trace functions. For every partition\footnote{This notation means that $\lambda$ is the tuple $(1,1,\cdots,2,2\cdots)$ consisting of $\lambda_1$ ones, $\lambda_2$ twos and so on.} $\lambda = 1^{\lambda_1}2^{\lambda_2}\cdots k^{\lambda_k}$ and every unitary matrix $U$, define
$$P_{\lambda}(U) = \prod_{j=1}^k \Tr(U^j)^{\lambda_j},$$
the mixed power trace function associated to $\lambda$. Applying Poisson summation, knowing the right hand side of \eqref{KSEqui} when $f=P_{\lambda}$ for all $\lambda$ allows us to compute the $n$-level densities of our family: a measure of the $n$-tuples of zeroes near the real-line. Hence, this paper will focus only on the functions $P_\lambda$.

The right hand side of \eqref{KSEqui} with $f = P_\lambda$ is well studied for the compact classical matrix Lie groups (\cite{DE,DS,Ram}). Therefore, the typical method of proving a statement as in \eqref{KSEqui} for a given family of $L$-function $\F$ is to use number theoretic tools to compute the left hand side and see which group $G$ matches.

The prototypical example of a family with symmetry type $USp$ is that of $L$-functions attached to quadratic Dirichlet characters whereas the prototypical example of a family with symmetry type $O$, $SO(even)$, $SO(odd)$ is that of $L$-functions attached to elliptic curves with mixed, positive and negative root numbers, respectively. See \cite{KS,Mil,Rud,You} for examples of these families and symmetry types.

These are examples of families of $L$-functions that we will say have ``quadratic structure"\footnote{Quadratic structure is not a well defined term. We use it merely to try and illustrate the difference between the more classical families of $L$-functions and the ones we wish to study here.}. If a family does not have ``quadratic structure" then we expect them to have unitary symmetry type. This is slightly disappointing as we would then see no difference in the statistics of the zeroes between two different families even if they may have differing ``higher order structures".

This paper is devoted to refining \eqref{KSEqui} for certain families with ``higher order structure". That is, by analyzing the multiplication of the scalars from the unit circle and utilizing the invariance of the Haar measure, one can show that  
\begin{align}\label{MatInt}
\int_{U(N)} P_\lambda(U) dU =0
\end{align}
for all partitions $\lambda$ (see \cite{DS,DE}). We then expect for all families of $L$-functions, $\F$, with no ``quadratic structure" that
\begin{align}\label{qlim0}
    \lim_{q\to\infty} \left\langle P_{\lambda}(\Theta_L)\right\rangle_{\F} =0.
\end{align}

Some question that can then be asked are:
\begin{enumerate}
\item Can we determine how fast \eqref{qlim0} tends to $0$ as $q$ tends to infinity?
\item Can we normalize \eqref{qlim0} in a natural way to get a non-zero limit?
\item If so, can we then express this non-zero limit as a matrix integral? 
\end{enumerate} 
We show in Theorem \ref{MainThm} the answer to these three question for the family of Dirichlet characters of fixed order is: yes. 

\subsection{Statement of Main Result}

Fix $r\geq 2$ an integer such that $q\equiv 1 \bmod{2r}$. For any $r$-th power free $G\in \Ff_q[T]$, let
\begin{align}\label{CharDef}
    \chi_G(F) = \left(\frac{F}{G}\right)_{r}
\end{align}
be the $r$-th power residue symbol for $\Ff_q[T]$ modulo $G$. Note that for this residue symbol to exist we need $q\equiv 1 \bmod{r}$. We impose the stronger condition $q\equiv 1 \bmod{2r}$ to make the $r$-th power reciprocity easier. Then we define the Dirichlet $L$-function attached to $\chi_G$ as  
\begin{align}\label{LFunDef}
L(u,\chi_G) := \sum_{F \mbox{ monic} } \chi_G(F) u^{\deg(F)}.
\end{align}
This will be a polynomial of degree $\deg(G)$. Hence $\Theta_L$ will be a $\deg(G)\times \deg(G)$ unitary matrix. The family we are interested in is then
\begin{align}\label{FamDef}
\mathcal{F}_{r}(N) := \{L(u,\chi_G) :  G \mbox{ is $r$-th power free}, \deg(G)=N\}
\end{align}
and we wish to determine the expected value of $P_\lambda(\Theta_L)$ as $L$ ranges over $\F_r(N)$.

Various statistics for the family quadratic characters ($r=2$) has been studied by many authors; see for example \cite{AK,BF,BF2,OS,Rud,Sound} as well as many others. The statistics for the family of higher order characters ($r>2$) is less well known. However, there has been some recent progress towards studying them \cite{BCDGL,CP,DFL,EP}. The author has already considered simpler statistics for a similar, more geometric, family in \cite{Meis} extending those result of \cite{BCDGL}.





\begin{thm}\label{MainThm}
For any $N\geq r\geq 2$ such that $q\equiv 1 \bmod{2r}$ and partition $\lambda= 1^{\lambda_1}2^{\lambda_2}\cdots k^{\lambda_k}$ such that $|\lambda| := \lambda_1+2\lambda_2+\cdots+k\lambda_k<\frac{rN}{2r-2}$, we get
$$\lim_{q\to\infty} \left\langle q^{\frac{r-2}{2r}|\lambda|} P_{\lambda}(\Theta_L)\right\rangle_{\mathcal{F}_{r}(N)} = \int_{U(N)} P_{\lambda}(U) \overline{\omega_r(U)} dU$$
where 
$$\omega_{r}(U) = \prod_{1\leq i_1<i_2<\dots <i_r \leq N} \left(1-x_{i_1}x_{i_2}\cdots x_{i_r}\right)^{(-1)^{r+1}}$$
and the $x_i$ are the eigenvalues of $U$. 
\end{thm}

\begin{rem}
The statement of Theorem \ref{MainThm} for $r=2$ may seem contradictory as it is known that in this case we should get a matrix integral over the symplectics on the right hand side \cite{KS,Rud}. However, there is no contradiction here as we show in Section \ref{QuadStruFam} that the matrix integral over the unitaries can be written as a matrix integral over the symplectics.
\end{rem}


The method to proving Theorem \ref{MainThm} is to use the theory of $L$-functions to write the expected value over $\F_r(N)$ combinatorially. This is written explicitly in Theorem \ref{CombThm}. We delegate writing the left hand side of Theorem \ref{MainThm} explicitly to Section \ref{CombSec} as the notation needed to write the theorem is outside the scope of an introduction. We then use results due to Diaconis and Evans \cite{DE} to construct a weight function such that Theorem \ref{MainThm} holds. Section \ref{WtFuncSec} is devoted to showing that this weight function is what we claim it to be.


This is now consistent with the philosophy that families with no quadratic structure should have unitary symmetry type. Indeed, when $r>2$, then it is reasonable to say that our family $\F_r(N)$ does not have quadratic structure as they are attached to characters of order higher than $2$. Moreover, we see that in the case $r>2$ then $q^{\frac{r-2}{2r}|\lambda|}$ tends to infinity with $q$ from which we may then conclude that the left hand side of \eqref{KSEqui} will be $0$ in this situation. This leads to an immediate Corollary.

\begin{cor}\label{UnitCor}
For any $N\geq r>2$ and partition $\lambda$ such that $|\lambda|<N$, we get
$$\lim_{q\to\infty} \left\langle P_\lambda(\Theta_L)\right\rangle _{\F_r(N)} = \int_{U(N)} P_\lambda(U)dU.$$
\end{cor}

The proof of Corollary \ref{UnitCor} is just to show that both sides of the equation therein is $0$. Note that we get an improved range of $|\lambda|<N$. This is due to the fact that the error term in the computation of Theorem \ref{MainThm} is bounded by some power of $q$. The bound of $|\lambda|$ is then taken so that this power is negative. Removing the normalization allows us to take $|\lambda|$ larger with the respective power still being negative.

\subsection{Quadratic Structure Families}\label{QuadStruFam}

We see that $\F_2(N)$ is one of the families we noted had ``quadratic structure". Therefore, we should expect it behaves more like either the symplectics or the orthogonals. Indeed, it is shown in \cite{KS} that for any continuous class function $f$, 
\begin{align}\label{KSr=2}
    \lim_{q\to \infty} \left\langle f(\Theta_L) \right\rangle_{\F_2(2N)} = \int_{USp(2N)} f(U)dU.
\end{align}

Since in the case of $r=2$, the normalization factor in Theorem \ref{MainThm} is just $1$, we see that the right hand side of Theorem \ref{MainThm} falls in the purview of \eqref{KSr=2}. In this light, it may seem odd then that we write the statistics of the family $\F_2(2N)$ in terms of the unitary matrices as it is know that they are all controlled by those of the  symplectics.

One conclusion that can be reached here is that the integral appearing in Theorem \ref{MainThm} and that in \eqref{KSr=2} are equal when $f=P_{\lambda}$ and $|\lambda| \leq 2N$. Indeed, one may prove exactly this directly without relying on either Theorem \ref{MainThm} or \eqref{KSr=2}.


\begin{thm}\label{ThmSymp}
For any positive integer $N$ and any partition $\lambda$ with $|\lambda|<2N$
$$\int_{U(2N)} P_\lambda(U) \overline{\omega_{Sp}(U)}  dU = \int_{USp(2N)} P_\lambda(U) dU$$
where $\omega_{Sp}(U) = \omega_2(U)$, as in Theorem \ref{MainThm}.
\end{thm}

This can easily be checked to be true as the right hand side was calculated by \cite{DS}. Then we can use the same method as in Section \ref{WtFuncSec} to use these results to construct a weight function that satisfies Theorem \ref{ThmSymp} and then prove that it has the appropriate form. 

As well as computing the expected value for $P_\lambda(U)$ for $USp(2N)$, the expected value over $O(N)$, the group of orthogonal matrices, was also computed in \cite{DS}. In the same way, we can then find a weight function that has the same property but for the orthogonal group. 

\begin{thm}\label{ThmOrth}
For any positive integer $N$ and any partition with $|\lambda|<N$
$$\int_{U(N)} P_{\lambda}(U) \overline{w_O(U)} dU = \int_{O(N)} P_{\lambda}(U) dU$$
where
$$w_O(U) = \prod_{1\leq i\leq j\leq N} \frac{1}{ 1-x_ix_j}$$
and the $x_i$ are the eigenvalues of $U$.
\end{thm}

Thus we see that the only difference in the symplectic weight function and the orthogonal weight function is whether we take the diagonal terms or not. In general, $w_r(U)$ is written as a product of strictly increasing $r$-tuples. The existence of families of $L$-functions with symplectic and orthogonal symmetry type then leads naturally into considering whether if we replace some strictly less than signs in $\omega_r(U)$ with less than or equal signs then are there natural families of $L$-functions whose statistics are governed by this new weight function? More precisely, for any $r$ and any $J\subset\{1,\dots,r-1\}$ consider the following weight functions
$$w_{J}(U) = \prod_{\substack{i_j =1,\cdots,N \\ i_j\leq i_{j+1}, j\in J \\ i_j< i_{j+1}, j\not\in J}} (1-x_{i_1}x_{i_2}\cdots x_{i_r})^{(-1)^{r+1}}.$$

\begin{ques}
For any $r$ and any $J\subset\{1,\dots,r-1\}$, does there exists a family of $L$-functions such that the normalized statistics tend to the unitaries weighted by the function $w_J$? 
\end{ques}

The answer in the case $J=\emptyset$ is clearly yes as we have just demonstrated that $\F_r(N)$ is such a family.

\textbf{Acknowledgements:} I would like to greatly thank Emilia Alvarez for the many fruitful conversations about random matrix theory over the months it took to finish this paper. I would also like to thank Alexander Lazar for directing me to useful references on combinatorics and Ze\'ev Rudnick for useful comments on an early draft.

\section{Heuristics and Quadratic Structure}

In this section, we will give heuristic arguments for why one should expect a result like Theorem \ref{MainThm} to be true and explain what is meant when we say a family of $L$-functions has ``quadratic structure". For simplicity of the heuristics, we will only work with the easiest power trace function: $\Tr(\Theta_L^n)$. Similar heuristics can be made about the general $P_\lambda$ functions using the analysis done in the proof of Theorem \ref{MainThm}.

\subsection{Euler Product and Trace Formula}

Recall that by the Riemann Hypothesis, for any $L$-function defined over $\Ff_q[T]$, we can find a unitary matrix $\Theta_L$ such that
$$L(u) = \det(1-\sqrt{q}u\Theta_L).$$
Taking the logarithmic derivative of the right hand side we get that
\begin{align}\label{logder1}
    \frac{d}{du} \log(L(u)) = \frac{d}{du} \log\left(\det(1-\sqrt{q}u\Theta_L)\right) = -\frac{1}{u} \sum_{n=1}^{\infty} q^{n/2} \Tr(\Theta_L^n)u^n.
\end{align}

On the other hand, if we take the logarithmic derivative of the Euler product formula in \eqref{EulProd}, we fine that
\begin{align}\label{logder2}
    \frac{d}{du} \log(L(u)) & = \sum_P \frac{d}{du} \log\left( 1 + a_L(P)u^{\deg(P)} + a_L(P^2)u^{2\deg(P)} + \cdots \right) \\
    & = \frac{1}{u}\sum_{F \mbox{ monic}} \Lambda(F) a_L^*(F) u^{\deg(F)}\nonumber
\end{align}
where 
\begin{align}\label{VMdef}
\Lambda(F) = \begin{cases} \deg(P) & F=P^k,  \mbox{ $P$ prime} \\ 0 & \mbox{otherwise} \end{cases}
\end{align}
is the function field von Mangoldt function.

Equating the coefficients of \eqref{logder1} and \eqref{logder2} we get a formula for the trace
\begin{align}\label{TrForm}
\Tr(\Theta_L^n) = -\frac{1}{q^{n/2}} \sum_{\deg(F)=n} \Lambda(F) a_L^*(F) = -\frac{1}{q^{n/2}} \sum_{d|n} d  \sum_{\deg(P)=d} a_L^*\left(P^\frac{n}{d}\right).
\end{align}

\subsection{Prime Polynomial Theorem}

In order to continue our heuristic, we need some facts about the number of primes in $\Ff_q[T]$ of degree $n$. The Prime Polynomial Theorem (see Proposition 2.1 of \cite{Rose}) states that
\begin{align}\label{PPT}
\sum_{\deg(F)=n} \Lambda(F) = q^n.
\end{align}

Now, if we denote
$$\pi_q(d) = \#\{ P\in \Ff_q[T] : P \mbox{ is prime and } \deg(P)=d \} $$
then we may rewrite \eqref{PPT} and conclude
$$q^n = \sum_{d|n} d\pi_q(d) \implies \pi_q(n) = \frac{q^n}{n} + O\left(q^{n/2}\right).$$
We can now use this asymptotic for $\pi_q(n)$ to prove a lemma that will be useful later.
\begin{lem}\label{PPPT}
For any $k$ and $n$, we have
$$\sum_{\deg(P)|n} \deg(P)^k = n^{k-1}q^n + O\left(n^kq^{\frac{n}{2}}\right)$$
where the sum is over prime polynomials.
\end{lem}
\begin{proof}
Indeed,
$$\sum_{\deg(P)|n} \deg(P)^k = \sum_{d|n} d^k\pi_q(d) = \sum_{d|n}\left( d^{k-1}q^d + O\left(d^kq^{d/2}\right)\right) = n^{k-1}q^n + O\left(n^kq^{\frac{n}{2}}\right)$$
\end{proof}

This now allows us to show that the contribution for when $d\geq 3$ in \eqref{TrForm} tends to $0$ as $q$ tends to infinity. That is, assuming $a^*_L(F)=O(1)$, we get
$$\frac{-1}{q^{n/2}} \sum_{\substack{d|n \\d\geq 3}} d \sum_{\deg(P)=d} a^*_L(P^{n/d}) \ll \frac{1}{q^{n/2}} \sum_{\substack{d|n \\ d\geq 3}} d\pi_q(d) \ll \frac{q^{n/3}}{q^{n/2}} = \frac{1}{q^{n/6}}.$$

Applying this, we see that as $q$ tends to infinity we are able to truncate \eqref{TrForm} to only consider the primes and prime squares  
\begin{align}\label{TrFormTrunc}
    \Tr(\Theta_L^n) = -\frac{n}{q^{n/2}} \left( \sum_{\deg(P)=n} a^*_L(P) + \frac{1}{2} \sum_{\deg(P)=n/2} a^*_L(P^2)  \right) + O\left(\frac{1}{q^{n/6}}\right)
\end{align}
where the second sum is understood to be empty if $n$ is odd. 

\subsection{Expected Value and Quadratic Structure}

Taking the expected value over our family we find that
$$\left\langle \Tr(\Theta_L^n)\right\rangle_{\F} = -\frac{n}{q^{n/2}}\left( \sum_{\deg(P)=n} \left\langle a^*_L(P)\right\rangle_{\F} + \frac{1}{2} \sum_{\deg(P)=n/2} \left\langle a^*_L(P^2)\right\rangle \right) + O\left(\frac{1}{q^{n/6}}\right).$$

For the primes we note that $a^*_L(P)=a_L(P)$, and so if we assume $a_L$ captures some arithmetic or geometric structure and $\F$ is a ``natural family" then the primes, being random, should be equidistributed among the structures. That is, it is natural to expect, at least for $n$ small enough\footnote{Note that one should not expect this limit to be $0$ for all $n$. In fact, determining if and for what values of $n$ this limit is not $0$ is crucial in determining the symmetry type of your family.}, that
$$\lim_{q\to\infty} \frac{n}{q^{n/2}}\sum_{\deg(P)=n} \left\langle a^*_L(P)\right\rangle_\F = 0.$$

Therefore, the expected value of $\Tr(\Theta_L^n)$ is determined precisely on what happens at the prime squares. We will then say that our family has ``quadratic structure" if
$$\lim_{q\to\infty} \frac{n}{q^{n/2}}\sum_{\deg(P)=n/2} \left\langle a^*_L(P^2)\right\rangle_\F \not=0$$
whenever $n$ is even. We see that in the above, our sum is of length $\approx \frac{q^{n/2}}{n}$. Therefore, for this limit to be non-zero we would need $\left\langle a^*_L(P^2)\right\rangle_\F$ to not equidistribute in the complex plane as the primes vary. That is the prime squares would not behave randomly.

\subsection{Higher Order Structures}

The first thing we did was trivially bound the primes that appeared with power higher than $2$. So, it begs the question: if a family has no quadratic structure then can the higher powers contribute? For instance, what happens if our family has ``cubic structure" so that $\left\langle a^*_L(P^3) \right\rangle_{\F}$ does not equidistribute in the complex plane as the primes vary? Then the prime cubes would not behave randomly and we could expect
$$\lim_{q\to\infty} \frac{n}{q^{n/3}} \sum_{\deg(P)=n/3} \left\langle a^*_L(P^3) \right\rangle_{\F} \not=0.$$
Hence, if one could show also that the contribution from the primes and prime squares tend to $0$ fast enough as $q$ tends to infinity, then one would be able to conclude that 
$$\lim_{q\to\infty}\left\langle q^{n/6}\Tr(\Theta_L^n)\right\rangle_{\F}$$
exists, is non-zero and, possibly, has a meaningful interpretation as a matrix integral.

\section{Trace Formula for our Family}\label{TraceSect}

Now we switch our attention to the specific families we are interested in: $\F_r(N)$. For ease of notation, we will define the set
\begin{align}\label{HrN}
\Hh_r(N) := \{G\in \Ff_q[T] :  G \mbox{ is $r$-th power free and } \deg(G)=N\}
\end{align}
so that our family becomes
$$\F_r(N) = \{L(u,\chi_G) : G\in\Hh_r(N)\}.$$
For further ease of notation, for any $G\in \Hh_r(N)$, we will define $\Theta_G$ as the Frobenius of $L(u,\chi_G)$.

\subsection{Formula for the Expected Value of a Normalized $P_{\lambda}$}

For any partition  $\lambda=1^{\lambda_1}2^{\lambda_2}\cdots k^{\lambda_k}$ we will equivalently write
$$\lambda = (n_1,n_2,\cdots,n_\ell)$$
where the first $\lambda_1$ of the $n_j$ are $1$, the next $\lambda_2$ of the $n_j$ are $2$ and so on. Then we may write
\begin{align}\label{Plambform1}
P_\lambda(\Theta_G) = \prod_{j=1}^{k} \Tr(\Theta_G^j)^{\lambda_j} = \prod_{j=1}^{\ell} \Tr(\Theta_G^{n_j}).
\end{align}
Since 
$$L(u,\chi_G) = \sum_{F \mbox{ monic}} \chi_G(F) u^{\deg(F)} = \prod_P \left( 1 - \chi_G(P)u^{\deg(P)} \right)^{-1}$$
we get that
$$\frac{d}{du} \log(L(u,\chi_G)) = \frac{1}{u}\sum_{F \mbox{ monic}} \Lambda(F) \chi_G(F) u^{\deg(F)}$$
and so by \eqref{TrForm}
\begin{align}\label{TrForm2}
\Tr(\Theta_G^n) = -\frac{1}{q^{\frac{n}{2}}} \sum_{\deg(F)=n} \Lambda(F) \chi_G(F) = -\frac{1}{q^{\frac{n}{2}}} \sum_{\deg(P)|n} \deg(P) \chi_G\left(P^{\frac{n}{\deg(P)}}\right).
\end{align}
Combining this with \eqref{Plambform1}, we obtain
\begin{align*}
    P_\lambda(\Theta_G) & = \prod_{j=1}^{\ell} \left(-\frac{1}{q^{\frac{n_j}{2}}} \sum_{\deg(P)|n_j} \deg(P) \chi_G\left(P^{\frac{n_j}{\deg(P)}}\right) \right) \\
    & = \frac{(-1)^{\ell}}{q^{\frac{|\lambda|}{2}}} \sum_{\substack{\deg(P_j)|n_j \\ j=1,\dots,\ell }} \prod_{j=1}^{\ell} \left(\deg(P_j) \chi_G\left(P_j^{\frac{n_j}{\deg(P_j)}}\right)\right).
\end{align*}
Where we have used the fact that if we write $\lambda = (n_1,n_2,\dots,n_\ell)$ then 
$$|\lambda|:= \lambda_1+2\lambda_2+\cdots+k\lambda_k = n_1+n_2+\cdots+n_\ell.$$

For ease of notation, for any tuple $(P_j) = (P_1,\dots,P_\ell)$ such that $\deg(P_j)|n_j$, we will write
\begin{align}\label{FPj}
    F(P_j) := \prod_{j=1}^{\ell} P_j^{\frac{n_j}{\deg(P_j)}}.
\end{align}

Further, since we are assuming $q\equiv 1 \bmod{2r}$, we have a simple form of $r$-th power reciprocity in that for any $F$ (see Theorem 3.3 of \cite{Rose})
$$\chi_G(F) = \left(\frac{F}{G}\right)_r = \left(\frac{G}{F}\right)_r = \chi_F(G).$$

Hence, averaging over $\Hh_r(N)$, we get
\begin{align}\label{Plambform2}
    \left\langle q^{\frac{r-2}{2r}|\lambda|} P_\lambda(\Theta_G)\right\rangle_{\Hh_r(N)} =  \frac{(-1)^{\ell}}{q^{\frac{|\lambda|}{r}}} \sum_{\substack{\deg(P_j)|n_j \\ j=1,\dots,\ell }} \left(\prod_{j=1}^{\ell} \deg(P_j)\right) \left\langle \chi_{F(P_j)}\left(G\right)\right\rangle_{\Hh_r(N)}
\end{align}

Finally, we get that $\chi_{F(P_j)}$ will be a non-trivial character if and only if $F(P_j)$ is not a perfect $r$-th power. Hence, we will define
\begin{align}\label{MainTerm}
MT_{\lambda}(N) := \frac{(-1)^{\ell}}{q^{\frac{|\lambda|}{r}}} \sum_{\substack{\deg(P_j)|n_j \\ j=1,\dots,\ell \\ F(P_j) = F^r }} \left(\prod_{j=1}^{\ell} \deg(P_j)\right) \left\langle \chi_{F^r}\left(G\right)\right\rangle_{\Hh_r(N)}.
\end{align}
and
\begin{align}\label{ErrTerm}
ET_{\lambda}(N) := \frac{(-1)^{\ell}}{q^{\frac{|\lambda|}{r}}} \sum_{\substack{\deg(P_j)|n_j \\ j=1,\dots,\ell \\ F(P_j) \not= F^r }} \left(\prod_{j=1}^{\ell} \deg(P_j)\right) \left\langle \chi_{F(P_j)}\left(G\right)\right\rangle_{\Hh_r(N)}.
\end{align}

\subsection{Coprimality Probability}

We see that if $F(P_j)=F^r$, then 
$$\chi_{F^r}(G) = \left(\frac{G}{F^r}\right)_r = \left(\frac{G}{F}\right)^r_r = \begin{cases} 1 & (F,G)=1 \\ 0 & (F,G)\not=1 \end{cases}.$$

Therefore,
$$\left\langle \chi_{F^r}(G) \right\rangle_{\Hh_r(N)} = \frac{1}{|\Hh_r(N)|} \sum_{G\in \Hh_r(N)} \chi_{F^r}(G) = \frac{|\Hh_r(N;F)|}{|\Hh_r(N)|} $$
where, we define
\begin{align}\label{HrNF}
\Hh_r(N;F) := \{G\in\Hh_r(N) : (G,F)=1\}.
\end{align}
So it remains to determine the size of $\Hh_r(N;F)$ for all $F$.

\begin{prop}\label{CoprSetSize}
For any $N\geq r\geq 2$ and any $F\in\Ff_q[T]$, we have
$$|\Hh_r(N;F)| = \frac{\phi(F)}{|F|} \prod_{P|F}\left(1+ \frac{1}{|P|^r-1}\right)  \left(q^N-q^{N+1-r}\right) + O(1).$$
where $\phi(F)$ is the Euler totient function. In particular, if we set $F=1$, then we get that
$$|\Hh_r(N)| = \begin{cases} q^N & N<r \\ q^N - q^{N+1-r} & N\geq r \end{cases}.$$
\end{prop}

\begin{proof}

Consider the generating series 
\begin{align*}
    \mathcal{G}_F(u) := \sum_{N=0}^{\infty} |\Hh_r(N;F)|u^N = \sideset{}{^*}\sum_{(G,F)=1} u^{\deg(G)}
\end{align*}
where the $^*$ indicates we take $G$ to be $r$-th power free. Then this generating series has an Euler product
\begin{align*}
    \mathcal{G}_F(u) & = \prod_{P\nmid F} \left(1 +u^{\deg(P)} + u^{2\deg(P)} + \cdots + u^{(r-1)\deg(P)} \right) \\
    & = \prod_{P\nmid F} \frac{1-u^{r\deg(P)}}{1-u^{\deg(P)}}\\
    & = \prod_{P|F} \frac{1-u^{\deg(P)}}{1-u^{r\deg(P)}} \frac{\zeta_q(u)}{\zeta_q(u^r)} \\
\end{align*}
where $\zeta_q(u)$ is the zeta function defined over $\Ff_q[T]$ given by the following equivalencies for $|u|<q^{-1}$
$$\zeta_q(u) :=\prod_P \left(1-u^{\deg(P)}\right)^{-1} =  \sum_F u^{\deg(F)} =  \frac{1}{1-qu}$$

Thus, we see that $\mathcal{G}_F(u)$ may be meromorphically extended to the region $|u|<1$ with a simple pole at $u=q^{-1}$. Thus, if $\Gamma = \{u : |u|=\frac{1}{2}\}$
\begin{align*}
    |\Hh_r(N;F)| & = -\Res_{u=q^{-1}}\left(\frac{\mathcal{G}_F(u)}{u^{N+1}}\right) + \frac{1}{2\pi i}\oint_{\Gamma} \frac{\mathcal{G}_F(u)}{u^{N+1}} du \\
    & = \prod_{P|F} \frac{1-\frac{1}{|P|}}{1-\frac{1}{|P|^r}} \left(q^N-q^{N+1-r}\right)+ O(1).
\end{align*}

To simplify, we note that
$$\prod_{P|F} 1-\frac{1}{|P|} = \prod_{P|F} \frac{|P|-1}{|P|} = \frac{\phi(F)}{|F|} \quad \quad \mbox{whereas} \quad \quad \frac{1}{1-\frac{1}{|P|^r}} = 1+\frac{1}{|P|^r-1}.$$

In particular, if $F=1$, then we can meromorphically extend $\mathcal{G}_1(u)$ to the whole complex plane. Hence if $\Gamma_d = \{u: |u|=d\}$, we get
\begin{align*}
    |\Hh_r(N)| & = -\Res_{u=q^{-1}}\left(\frac{\mathcal{G}_1(u)}{u^{N+1}}\right) + \frac{1}{2\pi i}\oint_{\Gamma_d} \frac{\mathcal{G}_1(u)}{u^{N+1}} du \\
    & = \begin{cases} q^N & N<r \\ q^N-q^{N+1-r} &N\geq r \end{cases} + O\left(\frac{1}{d^N}\right).
\end{align*}
Sending $d\to\infty$ concludes the proof.
\end{proof}

This leads to an immediate corollary for the expected value of $\chi_{F^r}(G)$ as $G$ runs over $\Hh_r(N)$.

\begin{cor}\label{CoprCor}
For any $F\in\Ff_q[T]$ and any $N\geq r\geq2$,
$$\left\langle \chi_{F^r}(G) \right\rangle_{\Hh_r(N)} = 1 + O\left(\frac{\tau(F)}{q} \right)$$
where $\tau(F)$ is the number of divisors of $F$.
\end{cor}

\begin{proof}

From Proposition \ref{CoprSetSize}, we get that
$$\left\langle \chi_{F^r}(G) \right\rangle_{\Hh_r(N)} =  \frac{|\Hh_r(N;F)|}{|\Hh_r(N)|} = \frac{\phi(F)}{|F|} \prod_{P|F} \left(1 + \frac{1}{|P|^r-1}\right) + O\left(\frac{1}{q^N}\right).$$
Further
$$\prod_{P|F} \left(1 + \frac{1}{|P|^r-1}\right) = 1 + \sum_{\substack{D|F \\ D\not=1 }} \mu^2(D) \prod_{P|D} \frac{1}{|P|^r-1} =  1 + O\left(\sum_{D\not=1} \frac{1}{|D|^r} \right) = 1+ O\left(\frac{1}{q^{r-1}}\right)  $$
and
$$\frac{\phi(F)}{|F|} = \prod_{P|F} 1-\frac{1}{|P|} = 1 + \sum_{\substack{D|F\\ D\not=1}} \frac{\mu(D)}{|D|} = 1 + O\left(\frac{\tau(F)}{q}\right).$$

\end{proof}

\begin{cor}\label{MainTermCor}
For any partition $\lambda=(n_1,n_2,\dots,n_\ell)$ and any $N\geq r\geq 2$ such that $q\equiv 1 \bmod{2r}$, we get
\begin{align}\label{MTN}
    MT_{\lambda}(N) = \frac{(-1)^{\ell}}{q^{\frac{|\lambda|}{r}}} \sum_{\substack{\deg(P_j)|n_j \\ j=1,\dots,\ell \\ F(P_j) = F^r }} \left(\prod_{j=1}^{\ell} \deg(P_j)\right) \left(1 + O\left(\frac{1}{q}\right)\right).
\end{align}
\end{cor}
\begin{proof}
Recall that 
$$F(P_j) = \prod_{j=1}^{\ell} P_j^{\frac{n_j}{\deg(P_j)}}$$
so that the result follows from Corollary \ref{CoprCor} and the fact that $\tau(F(P_j))=2^{\ell} = O(1)$ 
\end{proof}

Note that the right hand side \eqref{MTN} now does not depend on $N$. Hence we will define simply
\begin{align}\label{MT}
MT_{\lambda} := \frac{(-1)^{\ell}}{q^{\frac{|\lambda|}{r}}} \sum_{\substack{\deg(P_j)|n_j \\ j=1,\dots,\ell \\ F(P_j) = F^r }} \left(\prod_{j=1}^{\ell} \deg(P_j)\right).
\end{align}
Thus it remains to determine the limit at $q$ tends to infinity of $MT_{\lambda}$.






\subsection{Bounding $ET_\lambda(N)$}

\begin{prop}\label{ErrTermProp}
For any partition $\lambda$ and any $N\geq r\geq 2$ such that $q\equiv 1 \bmod{2r}$, we have
$$ET_\lambda(N) \ll q^{\left(1-\frac{1}{r}\right)|\lambda|-\frac{N}{2}}$$
\end{prop}

This leads to an immediate corollary
\begin{cor}
For any $N\geq r \geq 2$ such that $q\equiv1\bmod{2r}$ and any partition $\lambda$ such that $|\lambda|< \frac{rN}{2r-2}$, we have
$$\lim_{q\to\infty} ET_{\lambda}(N) =0$$
\end{cor}

We  first prove square root cancellation for $\left\langle\chi_F(G)\right\rangle_{\Hh_r(N)}$ in the case that $F$ is not an $r$-th power. 

\begin{lem}\label{ErrTermLem}
If $F$ is not an $r$-th power then
$$\left\langle\chi_F(G)\right\rangle_{\Hh_r(N)} \ll \frac{2^{\deg(F)}}{q^{N/2}}\left(1+\frac{\tau(F)}{q}\right)$$
where the implicit constant depends on $r$. 
\end{lem}

\begin{proof}

If $F$ is not an $r$-th power, then we let us consider the generating series
\begin{align*}
    \mathcal{G}(u,\chi_F) &= \sum_{N=0}^{\infty} \sum_{G\in \Hh_r(N)} \chi_F(G) u^{\deg(G)} \\
    & = \prod_P \left( 1 + \chi_F(P)u^{\deg(P)} + \chi_F(P^2)u^{2\deg(P)} + \cdots +\chi_F(P^{r-1}) u^{(r-1)\deg(P)} \right) \\
    & = \prod_{P\nmid F} \frac{1-u^{r\deg(P)}}{1-\chi_F(P)u^{\deg(P)} } \\
    & = \prod_{P|F} \left(\frac{1}{1-u^{r\deg(P)}}\right) \frac{L(u,\chi_F)}{\zeta_q(u^r)}.
\end{align*} 

Let $\Gamma = \{u: |u| = q^{-1/2}\}$. Then we see that $\mathcal{G}(u,\chi_F)$ is analytic in the region contained by $\Gamma$ so that
$$\sum_{G\in\Hh_r(N)} \chi_F(G) = \frac{1}{2\pi i} \oint_{\Gamma} \frac{\mathcal{G}(u,\chi_F)}{u^{N+1}}du \leq \max_{u\in \Gamma} \left|\frac{\mathcal{G}(u,\chi_F)}{u^{N+1}}\right| =q^{N/2} \max_{u\in \Gamma} |\mathcal{G}(u,\chi_F)| .$$
So it remains to bound $\mathcal{G}(u,\chi_F)$. Indeed we see that
$$\max_{u\in \Gamma} \left| \prod_{P|F} \frac{1}{1-u^{r\deg(P)}} \right| \leq \prod_{P|F} \frac{1}{1-q^{-\frac{r}{2}\deg(P)}} = \prod_{P|F} \left(1 + \frac{1}{q^{\frac{r}{2}\deg(P)} -1 } \right) = 1 +O\left(\frac{\tau(F)}{q}\right)$$
where the last estimate comes from a similar method as in Corollary \ref{CoprCor}. Further,
$$ \max_{u\in\Gamma} \left|\frac{1}{\zeta_q(u^r)}\right| = \max_{u\in\Gamma} \left|1-qu^r\right| \leq 1+q^{1-\frac{r}{2}} \ll 1+O\left(\frac{1}{q}\right).$$
Lastly, since we may write $L(u,\chi_F) = \det(1-\sqrt{q}u\Theta_F)$ for some $\deg(F)\times\deg(F)$ unitary matrix $\Theta_F$, we get that
$$\max_{u\in \Gamma} \left|L(u,\chi_F)\right| = \max_{u\in \Gamma} \left|\det(1-\sqrt{q}u\Theta_F)\right| \leq \left|\det(1-\Theta_F)\right| \leq 2^{\deg(F)}.$$

The result then follows from Proposition \ref{CoprSetSize}.

\end{proof}

\begin{proof}[Proof of Proposition \ref{ErrTermProp}]

Recall that
$$ET_{\lambda}(N) := \frac{(-1)^{\ell}}{q^{\frac{|\lambda|}{r}}} \sum_{\substack{\deg(P_j)|n_j \\ j=1,\dots,\ell \\ F(P_j) \not= F^r }} \left(\prod_{j=1}^{\ell} \deg(P_j)\right) \left\langle \chi_{F(P_j)}\left(G\right)\right\rangle_{\Hh_r(N)}  \mbox{ where }  F(P_j) := \prod_{j=1}^{\ell} P_j^{\frac{n_j}{\deg(P_j)}}$$
so that $\deg(F(P_j)) = |\lambda| = O(1)$ and $\tau(F) = 2^{\ell}=O(1)$. Then applying Lemma \ref{ErrTermLem}, we get
\begin{align*}
    ET_{\lambda}(N) & \ll \frac{1}{q^{\frac{\lambda}{r} + \frac{N}{2}}} \sum_{\substack{\deg(P_j)|n_j \\ j=1,\dots,\ell \\ F(P_j) \not= F^r }} \left(\prod_{j=1}^{\ell} \deg(P_j)\right) \\
    & \ll  \frac{1}{q^{\frac{\lambda}{r} + \frac{N}{2}}} \prod_{j=1}^{\ell} \left( \sum_{\deg(P)|n_j} \deg(P) \right) \\
    & = \frac{1}{q^{\frac{\lambda}{r} + \frac{N}{2}}} \prod_{j=1}^{\ell} q^{n_j} = q^{(1-\frac{1}{r})|\lambda|-\frac{N}{2}}
\end{align*}

\end{proof}

\section{Combinatorial Description of the Main Term}\label{CombSec}

In this section we will determine for which partitions $\lambda$ does 
$$\lim_{q\to\infty} MT_{\lambda} \not=0$$
and moreover give a combinatorial description of the limit.

\subsection{Set Partitions}
Recall that $MT_\lambda$ can be written as a sum over tuples of primes $(P_1,P_2,\dots,P_{\ell})$ such that $F(P_j) := \prod_{j=1}^{\ell} P_j^{\frac{n_j}{\deg(P_j)}}$ is an $r$-th power. Since the $P_j$ are primes this happens only when we can find disjoint subsets $J_1,\dots,J_t\subset \{1,\dots,\ell\}$ and primes $Q_1,\dots,Q_t$ such that
\begin{primenumerate}
    \item $\{1,\dots,\ell\} = \bigsqcup_{i=1}^{t} J_i$
    \item $P_j=Q_i$ for all $j\in J_i$
    \item $\sum_{j\in J_i} \frac{n_j}{\deg(P_j)} = \frac{\sum_{j\in J_i} n_j }{\deg(Q_i)} \equiv 0 \bmod{r}.$
\end{primenumerate}

Firstly, we see that condition $(1')$ is the definition for a set of subsets to form a \textit{set partition} of $\{1,\dots,\ell\}$. Secondly, we always have $\deg(P_j)|n_j$; hence if $Q_i=P_j$ for all $j\in J_i$, then $\deg(Q_i)|n_j$ for all $j$. Further, condition $(3')$ imposes the extra condition that $r | \sum_{j\in J_i} n_j$ and $\deg(Q_i)|\frac{1}{r}\sum_{j\in J_i}n_j$. Thus, we may replace conditions $(1')$, $(2')$, and $(3')$ with the following new conditions $(1)$, $(2)$ and $(3)$ 
\begin{enumerate}
    \item $J_1,\dots,J_t$ is a set partition of $\{1,\dots,\ell\}$;
    \item $\deg(Q_i)| \gcd(\frac{N_i}{r},\gcd_{j\in J_i} (n_j))$ for $i=1,\dots,t$;
    \item $N_i:=\sum_{j\in J_i} n_j \equiv 0 \bmod{r}$ for $i=1,\dots,t$.
\end{enumerate}

We would like to then write $MT_\lambda$ as a sum over all such sets and primes satisfying $(1)-(3)$. One obvious way to do this is to insist that the primes $Q_i$ are distinct. However, the notation needed quickly gets messy. A more convenient way to do it is to insist that the $J_i$ are as small as can be by the following fourth condition:
\begin{enumerate} \setcounter{enumi}{3}
    \item For any $J_i$ there does not exist a $J'_i \subsetneq J_i$ such that $\sum_{j\in J'_i} n_j \equiv 0 \bmod{r}.$
\end{enumerate}

Thus we get
\begin{align*}
    MT_\lambda & = \frac{(-1)^{\ell}}{q^{\frac{|\lambda|}{r}}} \sum_{\substack{\deg(P_j)|n_j \\ j=1,\dots,\ell \\ F(P_j) = F^r }} \prod_{j=1}^{\ell} \deg(P_j) =\frac{(-1)^{\ell}}{q^{\frac{|\lambda|}{r}}} \sum_{t=1}^{\ell} \sideset{}{^*}\sum_{\substack{J_i,Q_i \\ i=1,\dots,t}} \prod_{i=1}^t \deg(Q_i)^{|J_i|} 
\end{align*}
where the $^*$ indicates we are imposing the conditions $(1)-(4)$ above.

Before we continue with computing $MT_\lambda$, we will prove a brief Lemma showing that under the conditions imposed the sets $J_i$ have size bounded by $r$.

\begin{lem}\label{leqrLem}
For any sequence of positive integers $m_1,\dots,m_k$ such that $\sum_{j=1}^k m_j \equiv 0 \bmod{r}$, then we can find a subset $J\subset \{1,\dots,k\}$ with $|J|\leq r$ such that $\sum_{j\in J} m_j \equiv 0\bmod{r}$.
\end{lem}

\begin{proof}
Firstly, if $k\leq r$, then we take $J = \{1,\dots,k\}$ and the statement holds. Now if $k>r$, then consider the sums $s_i := \sum_{j=1}^i m_j$. By the pigeonhole principle, there must be an $1\leq i<i' \leq r+1$ such that $s_i\equiv s_{i'}\bmod{r}$. Hence taking $J = \{i+1,\dots,i'\}$  suffices.
\end{proof}

Now, for a fixed $J_1,\dots,J_t$, if we denote
$$g_i := \gcd\left(\frac{N_i}{r}, \gcd_{j\in J_i} (n_j) \right)$$
then we can now easily isolate condition $2$ and apply Lemma \ref{PPPT} to get that 
\begin{align*}
\sum_{\substack{\deg(Q_i) | g_i \\ i=1,\dots t }} \prod_{i=1}^t \deg(Q_i)^{|J_i|} & = \prod_{i=1}^t \sum_{\deg(Q)|g_i} \deg(Q)^{|J_i|} = \prod_{i=1}^t \left(g_i^{|J_i|-1}q^{g_i} + O\left(q^{\frac{g_i}{2}}\right)\right) \\
& = \prod_{i=1}^t g_i^{|J_i|-1}q^{g_i}\left(1+O\left(q^{-1/2}\right)\right)
\end{align*}

By definition we have that $g_i\leq \frac{N_i}{r}\leq\frac{|\lambda|}{r}$ and so if one $g_i\leq \frac{N_i}{r}-1$, then we get that
$$\prod_{i=1}^t g_i^{|J_i|-1} q^{g_i} \leq \frac{1}{q} \prod_{i=1}^t \left(\frac{N_i}{r}\right)^{|J_i|-1} q^{\frac{N_i}{r}} \leq \frac{1}{q} \left(\frac{|\lambda|}{r}\right)^\ell q^{\frac{|\lambda|}{r}}.$$
Thus the contribution from the case where one of the $g_i \leq \frac{N_i}{r}-1$ will be bounded by
\begin{align*}
    \frac{1}{q} \left(1+O(q^{-1/2})\right) \left(\frac{|\lambda|}{r}\right)^\ell \sum_{t=1}^{\ell} \sum_{\substack{J_1,\dots,J_t \\ (1),(3),(4)}} 1 \ll \frac{1}{q}
\end{align*}
where the implied constant depends only on $\lambda$. 

Conversely, if $g_i = \frac{N_i}{r}$ for all $i$, then we get
$$\prod_{i=1}^t g_i^{|J_i|-1}q^{g_i} = q^{\frac{|\lambda|}{r}} \prod_{i=1}^t \left(\frac{N_i}{r}\right)^{|J_i|-1}.$$
Therefore, if we create the last condition
\begin{enumerate}\setcounter{enumi}{4}
\item $g_i = \frac{N_i}{r}$ for $i=1,\dots,t.$
\end{enumerate}
Then we may conclude that 
\begin{align}\label{MTqlim}
    MT_\lambda = (-1)^{\ell} \sum_{t=1}^{\ell} \sideset{}{^{**}}\sum_{J_i} \prod_{i=1}^t \left(\frac{N_i}{r}\right)^{|J_i|-1} + O\left(q^{-1/2}\right)
\end{align}
where the $^{**}$ denotes that we are summing over all $J_1,\dots,J_t$ that satisfy conditions $(1),(3),(4)$ and $(5)$. 

We can now write down the left hand side of Theorem \ref{MainThm} in a maybe not so clear way.

\begin{prop}\label{SubsetProp}
For any $N\geq r\geq 2$ such that $q\equiv 1 \bmod{2r}$ and any partition $\lambda$ such that $|\lambda|<\frac{rN}{2r-2}$ then
$$\lim_{q\to\infty} \left\langle q^{\frac{r-2}{r}|\lambda|} P_\lambda(\Theta_L) \right\rangle_{\F_r(N)} = (-1)^{\ell} \sum_{t=1}^{\ell} \sideset{}{^{**}}\sum_{J_i} \prod_{i=1}^t \left(\frac{N_i}{r}\right)^{|J_i|-1}$$
where the notation for the right hand side in defined in the subsection above.
\end{prop}

\begin{proof}
Indeed, combining the work from Section \ref{TraceSect} with the above subsection we get
\begin{align*}
    \lim_{q\to\infty} \left\langle q^{\frac{r-2}{r}|\lambda|} P_\lambda(\Theta_L) \right\rangle_{\F_r(N)} & = \lim_{q\to\infty} \left(MT_\lambda(N)+ET_\lambda(N)\right) \\
    & = \lim_{q\to\infty} MT_{\lambda} \\ 
    & = (-1)^{\ell} \sum_{t=1}^{\ell} \sideset{}{^{**}}\sum_{J_i} \prod_{i=1}^t \left(\frac{N_i}{r}\right)^{|J_i|-1}
\end{align*}
\end{proof}


\subsection{Partition Decomposition}

As we saw in Lemma \ref{leqrLem}, conditions $(3)$ and $(4)$ imply that $|J_i|\leq r$. Now, we prove that condition $(5)$ implies that the $J_i$ must actually look like partitions of $r$.

\begin{lem}\label{PartLem1}
Let $m_1,\dots,m_k$ be positive integers such that 
\begin{enumerate}
    \item $M:=\sum_{j=1}^k m_j \equiv 0 \bmod{r}$
    \item $\frac{M}{r} = \gcd\left(\frac{M}{r}, m_1,\dots,m_k\right).$
\end{enumerate}
Then there must be positive integers $m'_1,\dots,m'_k$ such that $m'_1+\cdots+m'_k = r$ and $m_i = \frac{M}{r}m'_i$.
\end{lem}

\begin{proof}
Supposition $(2)$ implies that $\frac{M}{r} | m_i$ for all $i$ and so we can find positive integers $m'_i$ such that $m_i = \frac{M}{r} m'_i$. Now, we see that
$$M = \sum_{j=1}^k m_j = \frac{M}{r} \sum_{j=1}^k m'_j$$
which concludes the proof.
\end{proof}

The sets $J_i$  were constructed as the ways to decompose the partition $\lambda = (n_1,\dots,n_\ell)$ into subsets that contribute in the $q$-limit. Further, we have just now seen that these $J_i$ are themselves related to partitions of $r$. This then motivates us to talk about generating partitions out of a finite set of partition.

We can define an operation on partition in the following way. If we have two partitions, $\lambda,\mu$ such that  $\lambda = 1^{\lambda_1}2^{\lambda_2} \cdots k^{\lambda_k}$ and $\mu = 1^{\mu_1}2^{\mu_2}\cdots k^{\mu_k}$, then we define
$$\lambda \cdot \mu := 1^{\lambda_1+\mu_1} 2^{\lambda_2+\mu_2} \cdots k^{\lambda_k+\mu_k}.$$
By repetition we can then define for any positive integer, $a$
$$\lambda^a := \lambda\cdot\lambda\cdots\lambda = 1^{a\lambda_1}2^{a\lambda_2}\cdots k^{a\lambda_k}.$$
We extend this also to the case $a=0$ by setting $\lambda^0 = 1^02^0\cdots k^0 = (0)$, the $0$-partition. We can further define an action of a positive integer, $j$, by setting
$$j\lambda := j^{\lambda_1} (2j)^{\lambda_2} \cdots (kj)^{\lambda_k}.$$

Now, if we let $P_r$ be the set of partitions of $r$, we define now the set of partitions generated by $P_r$:
\begin{align}\label{PSr}
\WP_r := \left\{ \prod_{j=1}^{\infty} \prod_{\mu\in S_r} (j\mu)^{a_{j\mu}} : a_{j\mu}\geq 0 \mbox{ and only finitely many $a_{j\mu}\not=0$ }  \right\}.
\end{align}
Given any tuple $(a_{j\mu})_{j,\mu}$ where $j=1,\dots,\infty$, $\mu\in P_r$ and only finitely many $a_{j\mu}\not=0$, we define
$$\lambda(a_{j\mu}) = \lambda((a_{j\mu})_{j,\mu}) := \prod_{j=1}^{\infty} \prod_{\mu\in P_r} (j\mu)^{a_{j\mu}}.$$

Of course there may be multiple different tuples $(a_{j\mu})$ that give the same partition $\lambda$ in this way. For example, if $r=2$ then $P_2 = \{(1,1),(2)\}$ and if we set $a_{2(1,1)}=1$ and $a_{j\mu}=0$ otherwise  while we set $b_{(2)}=2$ and $b_{j\mu}=0$ otherwise, then we get
$$\lambda(a_{j\mu}) = (2,2) = \lambda(b_{j\mu}).$$

Before continuing, we will define one more bit of notation. If $\lambda = (n_1,n_2\dots,n_\ell)$, then we will denote $\{\lambda\} = \{n_1,n_2,\dots,n_\ell\}$ as the multiset consisting of the distinguished components of $\lambda$. Further, we will recall that two multisets are equal if and only if each element occurs in both multisets with equal multiplicity.  

\begin{lem}\label{PartSetCorr}
For a given partition $\lambda$ there is a correspondence between sets $J_1,\dots,J_t$ that satisfy conditions $(1)$, $(3)$, $(4)$ and $(5)$ and ways to write $\lambda = \lambda(a_{j\mu})$
\end{lem}

\begin{proof}
Indeed if we have such sets $J_1,\dots,J_t$ then Lemma \ref{PartLem1} implies that there exists an $M_i$ and partition $\mu_i$ of $r$ that 
$$\{n_j : j\in J_i\} = \{M_i\mu_i\}$$
are equal as multisets where we have set $\lambda=(n_1,n_2,\dots,n_\ell)$. Then by setting
$$a_{j\mu} = \#\{1\leq i \leq t : M_i=j, \mu_i=\mu\}$$ we get that $\lambda = \lambda(a_{j\mu})$. 

Conversely, if 
$$\lambda = \lambda(a_{j\mu}) = \prod_{j=1}^\infty \prod_{\mu\in S_r} (j\mu)^{a_{j\mu}}$$
then in particular, we have
\begin{align}\label{MultiSetEq}
\{\lambda\} = \bigcup_{j=1}^{\infty} \bigcup_{\mu\in S_r} \left(\{j\mu\} \cup \{j\mu\} \cup \cdots \cup \{j\mu\}\right)
\end{align}
are equal as multisets where there are $a_{j\mu}$ copies of $\{j\mu\}$. Thus if $\lambda= (n_1,n_2,\dots,n_\ell)$ then each copy of $\{j\mu\}$ will then be equal to $\{n_i : i\in J_{j\mu}\}$ for some $J_{j\mu}\subset\{1,\dots,\ell\}$. Further, since \eqref{MultiSetEq} is an equality of multisets, we may chose the $J_{j\mu}$ disjoint (so they satisfy condition $(1)$) and Lemma \ref{PartLem1} implies that they must also satisfy conditions $(3)$, $(4)$ and $(5)$.  
\end{proof}

We see that we can immediately determine the left hand side of Theorem \ref{MainThm} for a large class of partitions $\lambda$.

\begin{cor}
For any $N\geq r\geq 2$  such that $q\equiv 1 \bmod{2r}$ and partition $\lambda$ such that $|\lambda| < \frac{rN}{2r-2}$, then if $\lambda \not\in\WP_r$, we get
$$\lim_{q\to\infty} \left\langle q^{\frac{r-2}{2r}|\lambda|} P_\lambda(\Theta_L) \right\rangle_{\F_r(N)} = 0  $$
\end{cor}

\begin{proof}
Lemma \ref{PartSetCorr} shows that if $\lambda\not\in \WP_r$, then no such subsets $J_1,\dots,J_t$ can exist that satisfy conditions $(1)$, $(3)$, $(4)$ and $(5)$ and so the sum in Proposition \ref{SubsetProp} is empty and hence $0$. 
\end{proof}

\subsection{Types of Set Partitions}

Similarly to how if $\lambda\in \WP_r$, then there could be multiple ways to generate it with elements of $P_r$, there could also be multiple different sets $J_1,\dots,J_t$ that correspond to a given decomposition into partitions of $r$ as described in Lemma \ref{PartSetCorr}.

For example, if $r=3$ then $P_3 = \{(1,1,1),(1,2),(3)\}$. Setting $a_{(1,2)}=1$, $a_{2(1,1,1)}=1$ and $a_{j\lambda}=0$ otherwise, we see that
$$\lambda(a_{j\mu}) = (1,2,2,2,2).$$
So, setting $n_1=1$, $n_2=n_3=n_4=n_5=2$, we get $4$ different sets $J_1,J_2$ that give the same decomposition into partitions of $3$. That is, we can choose
\begin{align*}
J_1 = \{1,2\}, J_2 = \{3,4,5\} \mbox{ or } J_1 = \{1,3\}, J_2 = \{2,4,5\} \mbox{ or } \\
J_1 = \{1,4\}, J_2 = \{2,3,5\} \mbox{ or } J_1 = \{1,5\}, J_2 = \{2,3,4\}.
\end{align*}

With this and Lemma \ref{PartSetCorr} in mind, we define what it means for a set partition to be of a certain type.

\begin{defn}
For a given $\lambda = (n_1,\dots,n_\ell)$ we will say a set partition $J_1,\dots,J_t$ of $\{1,\dots,\ell\}$ is of \textbf{type $(a_{j\mu})_{j,\mu}$} if it corresponds to a way to write $\lambda = \lambda(a_{j\mu})$ as described in Lemma \ref{PartSetCorr}. 

More concretely, the set partition is of type $(a_{j\mu})_{j,\mu}$ if
$$a_{j\mu} = \#\left\{ 1\leq i \leq t : j|n_k \mbox{ for all } k\in J_i \mbox{ and } \left(\frac{n_k}{j} : k\in J_i\right)= \mu   \right\}.$$
\end{defn}

We can now use the notion of types of set partitions to rewrite Proposition \ref{SubsetProp}. That is, recalling the statement of Proposition \ref{SubsetProp}, we get that 
\begin{align*}
    \lim_{q\to\infty} \left\langle q^{\frac{r-2}{r}|\lambda|} P_\lambda(\Theta_L) \right\rangle_{\F_r(N)} & = (-1)^{\ell} \sum_{t=1}^{\ell} \sideset{}{^{**}}\sum_{J_i} \prod_{i=1}^t \left(\frac{N_i}{r}\right)^{|J_i|-1} \\
    & = (-1)^{\ell} \sum_{\substack{ (a_{j\mu})_{j,\mu} \\ \lambda(a_{j\mu}) = \lambda }} \sum_{t=1}^{\ell} \sum_{\substack{ J_1,\dots,J_t \\ \mbox{ of type } (a_{j\mu})_{j,\mu} }} \prod_{i=1}^t \left(\frac{N_i}{r}\right)^{|J_i|-1}.
\end{align*}

Now, if $J_1,\dots,J_t$ is of type $(a_{j\mu})_{j,\mu}$, then there will be $a_{j\mu}$ sets $J_i$ such that $\frac{N_i}{r} = j$ and $|J_i|=\ell(\mu)$, the length of $\mu$, for all $j=1,\dots,\infty$ and all $\mu\in P_r$. In particular, we get that
$$\prod_{i=1}^t \left(\frac{N_i}{r}\right)^{|J_i|-1} = \prod_{j=1}^{\infty} \prod_{\mu\in S_r} j^{(\ell(\mu)-1)a_{j\mu}}$$
depends only on the type of the set partition. Hence, we get the immediate Proposition.

\begin{prop}\label{PartProp}
For any $N\geq r\geq 2$ such that $q\equiv 1 \bmod{2r}$ and any partition $\lambda$ with $|\lambda|<\frac{rN}{2r-2}$,
$$\lim_{q\to\infty} \left\langle q^{\frac{r-2}{r}|\lambda|}P_\lambda(\Theta_L)  \right\rangle_{\F_r(N)} =  (-1)^{\ell} \sum_{\substack{ (a_{j\mu})_{j,\mu} \\ \lambda(a_{j\mu}) = \lambda }} C_r(a_{j\mu}) \prod_{j=1}^{\infty} \prod_{\mu\in P_r} j^{(\ell(\mu)-1)a_{j\mu}}$$
where $C_r(a_{j\mu})$ is the number of set partitions of $\{1,\dots,\ell\}$ of type $(a_{j\mu})_{j,\mu}$.
\end{prop}

\subsection{Combinatorial Statement}

Therefore, to find a combinatorial formula for the left hand side of Theorem \ref{MainThm} it remains to determine a combinatorial formula for the coefficients $C_r(a_{j\mu})$.

\begin{lem}\label{BijLem}
For any tuple $(a_{j\mu})_{j,\mu}$ such that only finitely many are non-zero and any partition $\lambda = \lambda(a_{j\mu})$ there is a 
$$\prod_{j=1}^{\infty} \prod_{\mu\in P_r} \prod_{k=1}^r (\mu_k!)^{a_{j\mu}} \mbox{ to } 1$$
correspondence from bijections of multisets
$$\phi: \{\lambda\} \to \prod_{j=1}^{\infty} \prod_{\mu\in P_r} \{j\mu\}^{a_{j\mu}}$$ 
and set partitions of type $(a_{j\mu})_{j,\mu}$
where for every $\mu\in P_r$, we write $\mu=1^{\mu_1}2^{\mu_2}\cdots r^{\mu_r}$.
\end{lem}

\begin{proof}
For any map $\phi$, define the sets
$$J_{j\mu,i}(\phi) = \left\{k : \phi(n_k) \mbox{ lies in the $i$-th copy of $\{j\mu\}$}\right\}$$
for all $j=1,\dots,\infty$, $\mu\in P_r$ and $i=1,\dots,a_{j\mu}$. Then this defines a set partition of $\{1,\dots,\ell\}$ of type $(a_{j\mu})_{j,\mu}$. Moreover, if $\phi$ and $\widetilde{\phi}$ are two bijections then we see that $J_{j\mu,i}(\phi) = J_{j\mu,i}(\widetilde{\phi})$ if and only if $\widetilde{\phi}\phi^{-1}$ fixes the $i$-th copy of $\{j\mu\}$. We see that if $\mu=1^{\mu_1}2^{\mu_2}\cdots r^{\mu_r}$ then there are
$$\prod_{k=1}^r \mu_k!$$
ways to fix $\{j\mu\}$. Therefore, two bijections correspond to the same set partition if and only if they fix each copy of $\{j\mu\}$ for each $j$ and each $\mu$ of which there are 
$$\prod_{j=1}^{\infty} \prod_{\mu\in P_r} \prod_{k=1}^r (\mu_k!)^{a_{j\mu}}$$
such maps.
\end{proof}

\begin{lem}\label{NumBijs}
The number of bijections of multisets as in  Lemma \ref{BijLem} is
$$\frac{\prod_{j=1}^{\infty}\lambda_j!}{\prod_{j=1}^{\infty} {\prod_{\mu\in P_r} a_{j\mu}!}}$$
\end{lem}

\begin{proof}
We may view every bijection $\phi$ as a bijection $\psi$ from $\{\lambda\}$ to itself of which there are $\prod_j \lambda_j!$ of them. Now each of these bijections $\psi$ correspond to the same bijection $\phi$ if and only if they permute the $a_{j\mu}$ copies of $\{j\mu\}$ for each $j=1,\dots,\infty$ and $\mu\in P_r$. There are $\prod_{j=1}^{\infty} \prod_{\mu\in P_r} a_{j\mu}!$ ways of doing this and the result follows.
\end{proof}

This now leads to us being able to write down a purely combinatorial formula for the left hand side of Theorem \ref{MainThm}.

\begin{thm}\label{CombThm}
For any $N\geq r\geq 2$ such that $q\equiv 1 \bmod{2r}$ and any partition $\lambda=1^{\lambda_1}2^{\lambda_2}\cdots$ with $|\lambda|<\frac{rN}{2r-2}$, then
$$\lim_{q\to\infty} \left\langle q^{\frac{r-2}{r}|\lambda|}P_\lambda(\Theta_L)  \right\rangle_{\F_r(N)} =  \prod_{j=1}^{\infty} \left((-1)^{\lambda_j} \lambda_j!\right) \sum_{\substack{ (a_{j\mu})_{j,\mu} \\ \lambda(a_{j\mu}) = \lambda }}  \prod_{j=1}^{\infty} \prod_{\mu\in P_r} \frac{1}{a_{j\mu}!} \left( \frac{j^{\ell(\mu)-1}}{\prod_{k=1}^r \mu_k!} \right)^{a_{j\mu}} $$
where for any $\mu\in S_r$, we write $\mu = 1^{\mu_1}2^{\mu_2} \cdots r^{\mu_r}$.
\end{thm}

\begin{proof}

Lemmas \ref{BijLem} and \ref{NumBijs} imply that
$$C_r(a_{j\mu}) = \prod_{j=1}^{\infty} \frac{\lambda_j!}{\prod_{\mu\in S_r} a_{j\mu}! \prod_{k=1}^{r} (\mu_k!)^{a_{j\mu}} }.$$
The result now follows from Proposition \ref{PartProp} and the fact that $\ell = \sum_{j=1}^{\infty} \lambda_j$.
\end{proof}

\section{Deriving the Weight Function}\label{WtFuncSec}

\subsection{Arbitrary Valued Weight Function}

In Theorem \ref{CombThm}, we determined a combinatorial formula for the left hand side of the equation in Theorem \ref{MainThm}. So it remains to show that the right hand side, with the given weight function, matches. We first prove a Proposition that shows how one can find a weight function that will give you any values.

\begin{prop}\label{WeightProp}
Given a real number $A_{\lambda}$ for every partition $\lambda$ then the weight function defined by
$$\omega(U) := \sum_{\lambda} \frac{A_{\lambda}}{z_{\lambda}} P_{\lambda}(U) \quad \quad \mbox{where} \quad \quad z_{\lambda} = \prod_{j=1}^{\infty} j^{\lambda_j} \lambda_j!,$$
and the sum is over all partitions written in the form $\lambda = 1^{\lambda_1}2^{\lambda_2}\cdots$, has the property that for every $|\lambda|\leq N$
$$\int_{U(N)} P_{\lambda}(U) \overline{\omega(U)} dU = A_{\lambda}.$$
\end{prop}

\begin{proof}

This follows almost immediately from Theorem 2.1 of \cite{DE} that states for any partitions $\lambda,\mu$ with $|\lambda|$ or $|\mu|\leq N$
\begin{align}\label{D-S}
\int_{U(N)} P_{\lambda}(U) \overline{P_{\mu}(U)} dU = \delta_{\lambda\mu} z_{\lambda}
\end{align}
where $\delta_{\lambda\mu}$ is $1$ if $\lambda=\mu$ and $0$ otherwise. Further, within the proof of Theorem 2.1 in \cite{DE} they show that the integral in \eqref{D-S} is always $0$ in the case that $|\mu|\not=|\lambda|$. Applying this result, we get that if $|\lambda|\leq N$ then
\begin{align*}
    \int_{U(N)} P_{\lambda}(U) \overline{\omega(U)} dU & = \sum_{\mu} \frac{A_\mu}{z_\mu} \int_{U(N)} P_\lambda(U) \overline{P_\mu(U)} dU \\
    & = \sum_{|\mu| =|\lambda| } \frac{A_\mu}{z_\mu} \delta_{\lambda\mu} z_\lambda \\
    & = A_\lambda
\end{align*}

\end{proof}

\subsection{Proof of Theorem \ref{MainThm}}

Theorem \ref{CombThm} and Proposition \ref{WeightProp} then implies that Theorem \ref{MainThm} is true with 
\begin{align}\label{omega1}
    \omega_{r}(U) & = \sum_{\lambda} \prod_{j=1}^{\infty} \left(\frac{-1}{j}\right)^{\lambda_j}\sum_{\substack{(a_{j\mu})_{j,\mu} \\ \lambda = \lambda(a_{j\mu}) }}  \prod_{j=1}^{\infty} \prod_{\mu\in P_r}\frac{1}{a_{j\mu}!} \left(\frac{j^{\ell(\mu)-1}}{\prod_{k=1}^r \mu_k!}\right)^{a_{j\mu}} P_\lambda(U)
\end{align}
and so it remains to show that this function has the desired equivalent definition. 

We first note that if $\lambda = \lambda(a_{j\mu})$ then 
\begin{align*}
    \{\lambda\} &= \{n_1,\dots,n_\ell\} = \bigcup_{j=1}^{\infty} \bigcup_{\mu\in P_r} \left(\{j\mu\}\cup\{j\mu\}\cup\dots\cup\{j\mu\} \right)
\end{align*}
so that if for every $\mu\in P_r$ we denote $\mu = (m_1(\mu),m_2(\mu),\cdots)$ then
$$P_{\lambda}(U) := \prod_{k=1}^{\ell} \Tr(U^{n_k})= \prod_{j=1}^{\infty}\prod_{\mu \in P_r} \left(\prod_{i=1}^{\ell(\mu)}  \Tr(U^{jm_i(\mu)})\right)^{a_{j\mu}} = \prod_{j=1}^{\infty} \prod_{\mu \in P_r} P^{a_{j\mu}}_{j\mu}(U).$$
Further, if we for every $\mu\in P_r$, we denote $\mu = 1^{\mu_1} 2^{\mu_2} \cdots r^{\mu_r}$ then we get that
$$\lambda = \lambda(a_{j\mu})= \prod_{j=1}^{\infty} \prod_{\mu\in P_r} (j\mu)^{a_{j\mu}} = \prod_{j=1}^{\infty} \prod_{\mu\in P_r} \prod_{k=1}^r (jk)^{\mu_ka_{j\mu}} = \prod_{J=1}^{\infty} J^{\sum_{k|J} \sum_{\mu\in P_r} \mu_k a_{\frac{J}{k}\mu} }.$$
Hence,
$$\lambda_j = \sum_{k|j} \sum_{\mu\in P_r} \mu_k a_{\frac{j}{k}\mu},$$
and therefore
\begin{align*}
    \prod_{j=1}^{\infty} \left(\frac{-1}{j}\right)^{\lambda_j} & = \prod_{j=1}^{\infty} \prod_{k|j} \sum_{\mu\in P_r} \left(\frac{-1}{j}\right)^{\mu_k a_{\frac{j}{k}\mu}} \\
    & = \prod_{j'=1}^{\infty} \prod_{k=1}^{r} \prod_{\mu\in P_r} \left(\frac{-1}{kj'} \right)^{\mu_k a_{j'\mu}} \\
    & = \prod_{j=1}^{\infty} \prod_{\mu\in P_r} \left(\frac{1}{j}\right)^{\ell(\mu)a_{j\mu}} \prod_{k=1}^r \left(\frac{-1}{k} \right)^{\mu_ka_{j\mu}}.
\end{align*}

For ease of notation, we will denote 
$$w(\mu) := \prod_{k=1}^r \frac{(-1)^{\mu_k}} {k^{\mu_k}\mu_k!}.$$
Plugging this into \eqref{omega1}, we find that
\begin{align*}
    \omega_r(U) & = \sum_{\lambda} \sum_{\substack{(a_{j\mu})_{j\mu} \\ \lambda = \lambda(a_{j\mu}) }} \prod_{j=1}^{\infty} \prod_{\mu\in P_r} \frac{1}{a_{j\mu}!} \left( \frac{w(\mu) P_{j\mu}(U)}{j} \right)^{a_{j\mu}} = \prod_{j=1}^{\infty} \prod_{\mu\in P_r} \sum_{a_{j\mu}=0}^{\infty}\frac{1}{a_{j\mu}!} \left(  \frac{w(\mu) P_{j\mu}(U)}{j} \right)^{a_{j\mu}} \\
    & = \prod_{j=1}^{\infty} \prod_{\mu\in P_r} \exp\left( \frac{w(\mu)P_{j\mu}(u)}{j} \right) =  \exp\left( \sum_{j=1}^{\infty} \frac{1}{j} \sum_{\mu\in P_r} w(\mu)P_{j\mu}(U) \right).
\end{align*}

We first note that 
\begin{align}\label{Pjmu}
P_{j\mu}(U) = \prod_{k=1}^r \Tr(U^{jk})^{\mu_k} = P_{\mu}(U^j)
\end{align}
so it is enough to solve the inner sum in the exponential for $j=1$ and $U$ arbitrary.

\begin{lem}\label{PartSumLem}
For any $r$ and any $N\times N$ unitary matrix $U$ with eigenvalues $x_1,\dots,x_N$, we have
$$\sum_{\mu\in P_r} w(\mu) P_{\mu}(U) = (-1)^r \sum_{1\leq i_1<\cdots<i_r\leq N } x_{i_1}x_{i_2}\cdots x_{i_r}.$$
Note that the right hand side is exactly the elementary symmetric polynomial $e_r(x_1,\dots,x_N)$.
\end{lem}

\begin{proof}
First we will calculate the sum over all partitions. That is,
\begin{align*}
    \sum_{\mu} w(\mu) P_\mu(U) & = \sum_{\mu_1,\mu_2,\cdots } \prod_{k=1}^{\infty} \frac{(-1)^{\mu_k}}{k^{\mu_k}\mu_k!} \Tr(U^k)^{\mu_k}  = \prod_{k=1}^{\infty} \sum_{\mu_k=0}^{\infty} \frac{1}{\mu_k!} \left(\frac{-\Tr(U^k)}{k}\right)^{\mu_k} \\
    & = \prod_{k=1}^{\infty} \exp\left( \frac{-\Tr(U^k)}{k} \right)  = \exp\left(-\sum _{k=1}^{\infty} \frac{\Tr(U^k)}{k}\right)
\end{align*}

Now, we find that
\begin{align*}
    -\sum_{k=1}^{\infty} \frac{\Tr(U^k)}{k} & = -\sum_{k=1}^{\infty} \frac{1}{k} \sum_{i=1}^N x_i^k  = - \sum_{i=1}^N \sum_{k=1}^{\infty} \frac{x_i^k}{k} \\
    & = \sum_{i=1}^N \log(1-x_i) = \log\left(\prod_{i=1}^k (1-x_i)\right)
\end{align*}

To conclude we get that 
$$\sum_{\mu} w(\mu)P_\mu(U) = \prod_{i=1}^N (1-x_i)$$
and $\sum_{\mu\in P_r} w(\mu)P_\mu(U)$ is the homogenous part of degree $r$, which is the elementary symmetric polynomial
$$e_r(x_1,\dots,x_N) = (-1)^r \sum_{1\leq i_1<\cdots<i_r\leq N} x_{i_1}x_{i_2}\cdots x_{i_r}.$$

\end{proof}

Applying \eqref{Pjmu} and Lemma \ref{PartSumLem} to the sum inside the exponential of $\omega_r(U)$, we find that
\begin{align*}
    \sum_{j=1}^{\infty} \frac{1}{j} \sum_{\mu\in P_r} \omega(\mu)P_{j\mu}(U) & = (-1)^r \sum_{j=1}^{\infty} \frac{1}{j} \sum_{1\leq i_1< \cdots < i_r\leq N} x^j_{i_1}x^j_{i_2}\cdots x^j_{i_r} \\
    & = (-1)^r \sum_{1<i_1<\cdots<i_r\leq N} \sum_{j=1}^{\infty} \frac{(x_{i_1}x_{i_2}\cdots x_{i_r})^j}{j} \\
    & = (-1)^{r+1} \sum_{1\leq i_1<\cdots<i_r\leq N} \log\left( 1 - x_{i_1}x_{i_2}\cdots x_{i_r} \right)
\end{align*}

Therefore, we conclude the proof of Theorem \ref{MainThm} with the final calculation that
\begin{align*}
    \omega_r(U) & = \exp\left(\sum_{j=1}^{\infty} \frac{1}{j} \sum_{\mu\in P_r} w(\mu)P_{j\mu}(U)\right) \\
    & = \exp\left((-1)^{r+1} \sum_{1\leq i_1<\cdots < i_r\leq N} \log(1-x_{i_1}x_{i_2}\cdots x_{i_r})\right)\\
    & = \prod_{1\leq i_1<\cdots <i_r\leq N} \left(1-x_{i_1}x_{i_2}\cdots x_{i_r}\right)^{(-1)^{r+1}}.
\end{align*}

\section{Symplectic and Orthogonal Weight Functions}

\subsection{Symplectic Weight Function}\label{SympWtSec}

Diaconis and Shahshahani \cite{DS} show that for a partition $\lambda=1^{\lambda_1}2^{\lambda_2}\cdots k^{\lambda_k}$ such that $\lambda_1+\lambda_2+\cdots+\lambda_k \leq N$, then
\begin{align}\label{DSSymp}
    \int_{USp(2N)} P_{\lambda}(U) dU = \prod_{j=1}^k (-1)^{(j-1)\lambda_j} g_j(\lambda_j)
\end{align}
where
\begin{align}\label{gjfunc}
g_j(\lambda) = \begin{cases} 0 & j \mbox{ and } \lambda \mbox{ is odd} \\ \left(\frac{j}{2}\right)^{\frac{\lambda}{2}}\frac{\lambda!}{(\lambda/2)!} & j \mbox{ is odd and } \lambda \mbox{ is even} \\ \sum_k \binom{\lambda}{2k} \left(\frac{j}{2}\right)^k \frac{(2k)!}{k!} & j \mbox{ is even} \end{cases}.
\end{align}
Diaconis and Evans \cite{DE} then sketch an argument as to how to extend this result to all $\lambda$ such that $|\lambda|\leq N$. 

\begin{rem}
We can now conclude that the well known product on the right hand side of \ref{DSSymp} can be written in terms of partitions of $2$ as in the right hand side of Theorem \ref{CombThm} by using Theroem \ref{MainThm} and the result of Katz and Sarnak mentioned in \eqref{KSr=2}. However, one is able to prove this directly in a fun and interesting combinatorial exercise. 
\end{rem}

Thus, we see that Theorem \ref{ThmSymp} follows from Proposition \ref{WeightProp} by setting 
\begin{align*}
    w_{Sp}(U)  = \sum_{\lambda} P_{\lambda}(U) \prod_{j=1}^{\infty} \frac{(-1)^{(j-1)\lambda_j}g_j(\lambda_j)}{j^{\lambda_j}\lambda_j!} 
\end{align*}
Noting that 
$$P_{\lambda}(U) = \prod_{j=1}^{\infty} \Tr(U^j)^{\lambda_j} = \prod_{j=1}^{\infty} P^{\lambda_j}_{(j)}(U),$$
where $(j)$ is the partition of $j$ consisting solely of $j$, we get that
\begin{align*}
    w_{Sp}(U) & = \sum_{\lambda} \prod_{j=1}^{\infty} \frac{(-1)^{(j-1)\lambda_j}g_j(\lambda_j)}{j^{\lambda_j}\lambda_j!} P^{\lambda_j}_{(j)}(U) \\
    & = \prod_{j=1}^{\infty} \sum_{\lambda_j=0}^{\infty} \frac{(-1)^{(j-1)\lambda_j}g_j(\lambda_j)}{j^{\lambda_j}\lambda_j!} P^{\lambda_j}_{(j)}(U). 
\end{align*}

Now, if $j$ is odd,
\begin{align*}
    \sum_{\lambda_j=0}^{\infty} \frac{(-1)^{(j-1)\lambda_j}g_j(\lambda_j)}{j^{\lambda_j}\lambda_j!} P^{\lambda_j}_{(j)}(U) & =  \sum_{\lambda_j=0}^{\infty} \frac{1}{\lambda_j!} \left(\frac{P^2_{(j)}(U)}{2j}\right)^{\lambda_j} = \exp\left( \frac{P^2_{(j)}(U)}{2j} \right) 
\end{align*}
while, if $j$ is even,
\begin{align*}
    \sum_{\lambda_j=0}^{\infty} \frac{(-1)^{(j-1)\lambda_j}g_j(\lambda_j)}{j^{\lambda_j}\lambda_j!} P^{\lambda_j}_{(j)}(U) & = \sum_{\lambda_j=0}^{\infty} \frac{(-1)^{\lambda_j}P^{\lambda_j}_{(j)}(U)}{j^{\lambda_j}\lambda_j!} \sum_{k=0}^{\infty} \binom{\lambda_j}{2k} \left(\frac{j}{2}\right)^k \frac{(2k)!}{k!}  \\
    & = \sum_{k=0}^{\infty} \frac{1}{k!} \left(\frac{j}{2}\right)^k \sum_{\lambda_j=2k}^{\infty}  \frac{1}{(\lambda_j-2k)!} \left(\frac{-P_{(j)}(U)}{j}\right)^{\lambda_j} \\
    & = \sum_{k=0}^{\infty} \frac{1}{k!} \left(\frac{P^2_{(j)}}{2j}\right)^k \sum_{\lambda_j=0}^{\infty} \frac{1}{\lambda_j!} \left(\frac{-P_{(j)}}{j}\right)^{\lambda_j} \\
    & = \exp\left( \frac{P^2_{(j)} - 2P_{(j)}}{2j} \right).
\end{align*}

Plugging this in, we then get
\begin{align*}
    w_{Sp}(U) & = \prod_{j \mbox{ odd}} \exp\left( \frac{P^2_{(j)}(U)}{2j} \right)  \prod_{j \mbox{ even}} \exp\left( \frac{P^2_{(j)} - 2P_{(j)}}{2j} \right) \\
    & = \exp\left( \frac{1}{2} \sum_{j=1}^{\infty} \frac{P_{(j)}^2(U) - P_{(2j)}(U)}{j} \right).
\end{align*}

Finally, we get find that
\begin{align*}
    P_{(j)}^2(U) - P_{(2j)}(U) = \left(\sum_{i=1}^N x^j_i\right)^2 - \sum_{i=1}^N x_i^{2j} = 2 \sum_{1\leq i<k\leq N} x^j_ix_k^j
\end{align*}
to conclude that
\begin{align*}
    w_{Sp}(U) & = \exp\left( \sum_{1\leq i<k\leq N} \sum_{j=1}^{\infty} \frac{(x_ix_k)^j}{j} \right) \\
    & = \exp\left( - \sum_{1\leq i<k\leq N} \log(1-x_ix_k) \right) \\
    & = \prod_{1\leq i < k \leq N} \frac{1}{1-x_ix_k} \\
    & = w_2(U)
\end{align*}
which completes the proof of Theorem \ref{ThmSymp}.

\subsection{Orthogonal Weight Function}\label{OrthWtSec}

Similarly Diaconis and Shahshahani \cite{DS} and then extended by Diaconis and Evans \cite{DE} showed that for every partition $\lambda$ with $|\lambda|\leq N$, we have
\begin{align}\label{DSORth}
    \int_{USp(2N)} P_{\lambda}(U) dU = \prod_{j=1}^k  g_j(\lambda_j)
\end{align}
where $g_j(\lambda)$ is as defined in \ref{gjfunc}. Hence, the exact same computations as in Section \ref{SympWtSec} will carry over. The only difference is that when $j$ is even we will get a factor of
$$\exp\left( \frac{P_{(j)}^2+2P_{(j)}}{2j} \right).$$

Thus we get
$$w_O(U) = \exp\left( \frac{1}{2} \sum_{j=1}^{\infty} \frac{P_{(j)}^2(U) + P_{(2j)}(U)}{j} \right)$$
and
\begin{align*}
    P_{(j)}^2(U) + P_{(2j)}(U) = \left(\sum_{i=1}^N x^j_i\right)^2 + \sum_{i=1}^N x_i^{2j} = 2 \sum_{1\leq i\leq k\leq N} x^j_ix_k^j
\end{align*}
from whence we may conclude
$$w_O(U) = \prod_{1\leq i\leq k\leq N} \frac{1}{1-x_ix_k}$$
proving Theorem \ref{ThmOrth}.

\bibliography{main}

\newcommand{\etalchar}[1]{$^{#1}$}
\begin{thebibliography}{BCD{\etalchar{+}}18}

\bibitem[AK12]{AK}
JC~Andrade and JP~Keating.
\newblock The mean value of l (12, $\chi$) in the hyperelliptic ensemble.
\newblock {\em Journal of Number Theory}, 132(12):2793--2816, 2012.

\bibitem[BCD{\etalchar{+}}18]{BCDGL}
Alina Bucur, Edgar Costa, Chantal David, Joao Guerreiro, and DAVID LOWRY-DUDA.
\newblock Traces, high powers and one level density for families of curves over
  finite fields.
\newblock In {\em Mathematical Proceedings of the Cambridge Philosophical
  Society}, volume 165, pages 225--248. Cambridge University Press, 2018.

\bibitem[BF18a]{BF}
H~Bui and A~Florea.
\newblock Zeros of quadratic dirichlet l-functions in the hyperelliptic
  ensemble.
\newblock {\em Transactions of the American Mathematical Society},
  370(11):8013--8045, 2018.

\bibitem[BF18b]{BF2}
HM~Bui and A~Florea.
\newblock Hybrid euler-hadamard product for quadratic dirichlet l--functions in
  function fields.
\newblock {\em Proceedings of the London Mathematical Society}, 117(1):65--99,
  2018.

\bibitem[CP19]{CP}
Peter~J Cho and Jeongho Park.
\newblock Low-lying zeros of cubic dirichlet l-functions and the ratios
  conjecture.
\newblock {\em Journal of Mathematical Analysis and Applications},
  474(2):876--892, 2019.

\bibitem[DE01]{DE}
P~Diaconis and S~Evans.
\newblock Linear functionals of eigenvalues of random matrices.
\newblock {\em Transactions of the American Mathematical Society},
  353(7):2615--2633, 2001.

\bibitem[DFL19]{DFL}
Chantal David, Alexandra Florea, and Matilde Lalin.
\newblock The mean values of cubic l-functions over function fields.
\newblock {\em arXiv preprint arXiv:1901.00817}, 2019.

\bibitem[DS94]{DS}
P~Diaconis and M~Shahshahani.
\newblock On the eigenvalues of random matrices.
\newblock {\em Journal of Applied Probability}, 31(A):49--62, 1994.

\bibitem[EP21]{EP}
Alexei Entin and Noam Pirani.
\newblock Local statistics for zeros of artin-schreier l-functions.
\newblock {\em arXiv preprint arXiv:2107.02131}, 2021.

\bibitem[KS99]{KS}
N~Katz and P~Sarnak.
\newblock {\em Random matrices, Frobenius eigenvalues, and monodromy},
  volume~45.
\newblock American Mathematical Soc., 1999.

\bibitem[Mei20]{Meis}
Patrick Meisner.
\newblock Lower order terms for expected value of traces of frobenius of a
  family of cyclic covers of $\mathbb{P}^1_{\mathbb{f}_q}$ and one-level
  densities.
\newblock {\em arXiv preprint arXiv:2006.16886}, 2020.

\bibitem[Mil04]{Mil}
S~Miller.
\newblock One-and two-level densities for rational families of elliptic curves:
  evidence for the underlying group symmetries.
\newblock {\em Compositio Mathematica}, 140(4):952--992, 2004.

\bibitem[{\"O}S93]{OS}
Ali~E {\"O}zl{\"u}k and Chip Snyder.
\newblock Small zeros of quadratic l-functions.
\newblock {\em Bulletin of the Australian Mathematical Society},
  47(2):307--319, 1993.

\bibitem[Ram95]{Ram}
A~Ram.
\newblock Characters of brauer’s centralizer algebras.
\newblock {\em Pacific journal of Mathematics}, 169(1):173--200, 1995.

\bibitem[Ros02]{Rose}
Michael Rosen.
\newblock {\em Number theory in function fields}, volume 210.
\newblock Springer Science \& Business Media, 2002.

\bibitem[Rud10]{Rud}
Ze{\'e}v Rudnick.
\newblock Traces of high powers of the frobenius class in the hyperelliptic
  ensemble.
\newblock {\em Acta Arithmetica}, 143(1):81--99, 2010.

\bibitem[Sou00]{Sound}
Kannan Soundararajan.
\newblock Nonvanishing of quadratic dirichlet l-functions at s= 1/2.
\newblock {\em Annals of Mathematics}, 152(2):447--488, 2000.

\bibitem[SST16]{SST}
P~Sarnak, SW~Shin, and N~Templier.
\newblock Families of l-functions and their symmetry.
\newblock In {\em Families of automorphic forms and the trace formula}, pages
  531--578. Springer, 2016.

\bibitem[Wei48]{wei}
A~Weil.
\newblock Sur les courbes al{\'e}gbriques et les vari{\'e}t{\'e}s qui s'en
  d{\'e}duisent.
\newblock {\em Actualiti{\'e}s Sci. Ind.}, (1041), 1948.

\bibitem[You06]{You}
M~Young.
\newblock Low-lying zeros of families of elliptic curves.
\newblock {\em Journal of the American Mathematical Society}, 19(1):205--250,
  2006.

\end{thebibliography}
\bibliographystyle{alpha}

\end{document}